\documentclass{amsart}
    \usepackage{amsmath}
    \usepackage{amssymb,amsfonts}
    \usepackage[all,arc]{xy}
    \usepackage{enumerate}
    \usepackage{mathrsfs}
    \usepackage{amsthm}
    \usepackage{enumitem}
    \usepackage{extarrows}
    \usepackage{verbatim}
    \usepackage{dsfont}
    \usepackage{tikz-cd}
    \usepackage{tensor}
    \usepackage{mathtools}

    \newtheorem{thm}{Theorem}[section]
    \newtheorem{cor}[thm]{Corollary}
    \newtheorem{prop}[thm]{Proposition}
    \newtheorem{lem}[thm]{Lemma}

    \theoremstyle{definition}
    \newtheorem{defn}[thm]{Definition}

    \theoremstyle{remark}
    
    \newtheorem{rem}[thm]{Remark}

    \newcommand{\C}{\mathbb{C}}
    
    \newcommand{\Q}{\mathbb{Q}}
    \newcommand{\R}{\mathbb{R}}
    \newcommand{\Z}{\mathbb{Z}}
    
    \newcommand{\pP}{\mathbb{P}}

    \newcommand{\cO}{\mathcal{O}}
    \newcommand{\cI}{\mathcal{I}}

    \renewcommand{\H}{\mathrm H}
    \newcommand{\cH}{\mathcal{H}}

    \newcommand{\Coh}{\mathrm{Coh}}

    \newcommand{\Ku}{\mathsf{Ku}}

    \newcommand{\Hom}{\mathrm{Hom}}
    \newcommand{\Ext}{\mathrm{Ext}}
    
    \newcommand{\half}{\frac{1}{2}}
    \newcommand{\ch}{\mathrm{ch}}
    \newcommand{\Pic}{\mathrm{Pic}}
    \makeatletter
    
    \newcommand{\cE}{\mathcal{E}}
    \newcommand{\tcE}{\tilde{\mathcal{E}}}
    \newcommand{\tE}{\tilde{E}}
    
    \newcommand*{\rom}[1]{\expandafter\@slowromancap\romannumeral #1@}
    
    \makeatletter
    
    \let\c@equation\c@thm
    \makeatother
    \numberwithin{equation}{section}

\title{Bridgeland stability of minimal instanton bundles on Fano threefolds }
\author{Xuqiang QIN}
\address{Department of Mathematics, The University of North Carolina at Chapel Hill, 205 S Columbia St, Chapel Hill, NC 27514, USA}
\email{qinx@unc.edu}


\begin{document}

\begin{abstract}
    We prove that minimal instanton bundles on a Fano threefold $X$ of Picard rank one and index two are semistable objects in the Kuznetsov component $\mathsf{Ku}(X)$, with respect to the stability conditions constructed by Bayer, Lahoz, Macr\`i and Stellari. When the degree of $X$ is at least $3$, we show torsion free generalizations of minimal instantons are also semistable objects. As a result, we describe the moduli space of semistable objects with same numerical classes as minimal instantons in $\mathsf{Ku}(X)$. We also investigate the stability of acyclic extensions of non-minimal instantons.
\end{abstract}

\maketitle

\setcounter{tocdepth}{1}
\tableofcontents
\section{Introduction}
Instanton bundles first appeared on the $4$-sphere $S^4$ as a way to describe Yang-Mills instantons. They serve as bridges between algebraic geometry and mathematical physics. The notion of mathematical instanton bundles was first introduced on $\mathbb{P}^3$, then generalized to Fano threefolds by Faenzi\cite{Fa} and  Kuznetsov\cite{Ku2}.
\begin{defn}\cite{Ku2}
    Let $X$ be a Fano threefold of Picard rank $1$ and index $2$. An \emph{instanton bundle of charge $n$} on $X$ is a stable vector bundle $E$ of rank $2$ with $c_1(E)=0, c_2(E)=n$, enjoying the instantonic vanishing condition: 
    \begin{align*}
        H^1(X,E(-1))=0.
    \end{align*}
\end{defn}
Any instanton bundle $E$ will have charge $c_2(E)\geqslant 2$. The instanton bundles of charge $2$ are called the \emph{minimal instantons}. By definition, the moduli space of minimal instanton bundles on $X$ is an open subscheme of the moduli space $M$ of Gieseker-semistable rank $2$ sheaves with $c_1=0, c_2=2$ and $c_3=0$. When $X$ is a cubic threefold, \cite{D} classified sheaves in $M$ and described $M$ as a blow-up of the intermediate Jacobian of $X$. His work was generalized to $X$ of degree $5$ and $4$ by the author\cite{Q1}\cite{Q2}.\\

On the other hand, Bridgeland\cite{Br07} introduced the notion of stability conditions on a triangulated category. He showed that the set parametrizing stability conditions on a triangulated category has a natural structure of a complex manifold. Since then stabilities on the bounded derived category $\mathrm{D^b}(X)$ of a smooth projective variety $X$ have been intensely studied. \cite{Br07}\cite{Ma}\cite{Ok} gave complete descriptions of the stability manifold $\mathrm{Stab}(X)$ when $X$ is a smooth projective curve. When $X$ is a smooth projective surface, stability conditions and related moduli spaces were studied in \cite{Br08}\cite{ABCH}\cite{AB}\cite{LiZ}\cite{Nu}, among many other papers. When the dimension of $X$ is at least three, the construction of Bridgeland stability conditions becomes challenging. We refer the readers to \cite{BMT}\cite{BMSZ}\cite{Li} for information on stabilities on Fano threefolds.

Bayer, Lahoz, Macr\`i and Stellari\cite{BLMS} provided a criterion to define stability condition on the right orthogonal complements of an exceptional collection in a triangulated category. They applied the criterion to Fano threefolds of Picard rank one and cubic fourfolds and induced stability conditions on them. See \cite{LLMS}\cite{LPZ1}\cite{LPZ2} for some applications on cubic fourfolds. For our purpose, let $X$ be a Fano threefold of Picard rank one and index two. The bounded derived category $\mathrm{D^b}(X)$ has the following semiorthogonal decomposition:
\begin{align*}
    \mathrm{D^b}(X)=\langle\mathsf{Ku}(X),\cO_X,\cO_X(H)\rangle
\end{align*}
where $H$ is the ample generator of the Picard group. The triangulated subcategory $\mathsf{Ku}(X)$ is called the \emph{Kuznetsov component}. Explicit computations in \cite{PY} shows that the criterion of \cite{BLMS} induces a family $\sigma(\alpha,\beta)$ of stability conditions on $\mathsf{Ku}(X)$ where $(\alpha,\beta)$ lies in a triangular region in the half plane $\R_{>0}\times \R$. Moreover, \cite{PY} showed that the family lies in a single orbit $\mathcal{K}$ with respect to the standard action of $\tilde{\mathrm{GL}}_2^+(\R)$ on the stability manifold $\mathrm{Stab}(\mathsf{Ku}(X))$ of $\mathsf{Ku}(X)$.

Pertusi and Yang\cite{PY} showed that ideal sheaves of lines  on $X$ (which are easily checked to belong to $\mathsf{Ku}(X)$) are stable objects with respect to any $\sigma\in\mathcal{K}$. Using this, they were able to identify the Fano surface of lines on $X$ with (if the degree of $X$ is $1$, an irreducible component of) the moduli space of $\sigma(\alpha,\beta)$-stable objects with the numerical class $[\cI_l]$ in $\mathsf{Ku}(X)$. 

Minimal instantons are known to be objects in $\mathsf{Ku}(X)$ and they have twice the numerical class of $[\cI_l]$. It is natural to ask if one can generalize the results of Pertusi and Yang\cite{PY} to minimal instantons. Our first result establishes their (semi)stability.
\begin{thm}\label{main1}
    Let $X$ be a Fano threefold of Picard rank one, index two and degree $d$. Let $E$ be a minimal instanton bundle on $X$. Then $E$ is $\sigma$-semistable for any $\sigma\in\mathcal{K}$. If $d\geqslant2$, then $E$ is $\sigma$-stable. 
\end{thm}

When $d\geqslant3$, the classifications in \cite{D}\cite{Q1}\cite{Q2} showed that torsion free (but not locally free) generalizations of minimal instanton bundles are also objects in $\mathsf{Ku}(X)$. Our second result compares the moduli space of Gieseker-semistable rank $2$ sheaves with $c_1=0, c_2=2, c_3=0$  with the moduli space of $\sigma$-semistable objects of numerical class $2[\cI_l]$ in $\mathsf{Ku}(X)$. 
\begin{thm}\label{main2}
    Let $X$ be a Fano threefold of Picard rank one, index two and degree $d$. If $d\geqslant3$, then for any $\sigma\in\mathcal{K}$, the moduli space $M_d$ of Gieseker-semistable sheaves on $X$ with Chern character $(2,0,-2,0)$ and satisfying $H^1(E(-1))=0$ is isomorphic to a moduli space $M_\sigma(\mathsf{Ku}(X),2[\cI_l])$ of $\sigma$-semistable objects in $\mathsf{Ku}(X)$ with numerical class twice of that of an ideal sheaf of a line in $X$.
\end{thm}
As a result, $M_\sigma(\mathsf{Ku}(X),2[\cI_l])$ is a smooth projective variety of dimension $5$ by the description of $M_d$ in \cite{D}\cite{Q1}\cite{Q2} for $d\geqslant 3$ (see Section $2$). We mention that by \cite{D}\cite{Q2}, the vanishing $H^1(E(-1))=0$ is satisfied by any Gieseker-semistable sheaves $E$ with Chern character $(2,0,-2,0)$ for $d=3$ and $4$.\ 

We note that most of our arguments for Theorem \ref{main2} work for $d=1,2$. However there is no complete description of $M_2$ or $M_1$ to the author's knowledge.\ 

In the last section, we consider non-minimal instantons. Kuznetsov\cite{Ku2} showed that one can associate to such an instanton a unique vector bundle in $\Ku(X)$, called its \emph{acyclic extension} (for details see Section 6). We establish the (semi)stability of these acyclic extensions of instantons with charge $3$.
\begin{thm}
    Assume $d\neq 1$. Let $\tcE$ be the acyclic extension of an instanton bundle $\cE$ of charge $3$. Then $\tcE$ is $\sigma$-semistable for any stability condition $\sigma\in\mathcal{K}$. If $d\geqslant3$, $\tcE$ is $\sigma$-stable.
\end{thm}

\subsection*{Other related work}
Lahoz, Macr\`i and Stellari\cite{LMS} constructed the first examples of Bridgeland stability conditions on the Kuznetsov component of a cubic threefold. They proved that the moduli space of semistable objects with numerical class $2[\cI_l]$ is isomorphic to $M_3$ for properly chosen stabilities. We note it is not known whether the stability constructed there is in $\mathcal{K}$.

In a recent work \cite{BBF et}, the authors studied the moduli space of semistable objects with numerical class $[\cI_l]+[S(\cI_l)]$ with respect to some $\sigma\in\mathcal{K}$ on a cubic threefold, where $S$ is the Serre functor of $\mathsf{Ku}(X)$. They showed the moduli space is isomorphic to an blow-up of the theta divisor of the intermediate Jacobian.

Petkovic and Rota\cite{PR} classified the stable objects in the moduli space containing the Fano surface of lines when $d=1$.

During the completion of this paper, the author was made aware by Zhiyu Liu and Shizhuo Zhang of their independent preprint\cite{LZ} in which they claim similar results to Theorem \ref{main1} and \ref{main2}.

\subsection*{Acknowledgement}
I am very grateful to Laura Pertusi for answering my questions on Bridgeland stabilities, for many interesting discussions, as well as her comments on an early draft of this paper. I would like to thank Justin Sawon for interesting discussions and his support. I thank Zhiyu Liu and Shizhuo Zhang for informing me of their result and sending their draft. I am very grateful to an anonymous referee for suggesting Lemma 6.4 and its proof. Finally, I thank the referees for careful reading of the paper and useful suggestions.


\section{Review on Fano threefolds and minimal instanton bundles}
\subsection{Fano threefolds of Picard rank $1$ and index $2$}
A Fano variety $X$ is a smooth projective variety whose anticanonical divisor $-K_X$ is ample. Its index $i(X)$ is defined to be the largest integer so that $-K_X=i(X)H$ for some ample divisor $H$. It is well-known that $i(X)\leqslant \dim(X)+1$ (see \cite{IP}), with equality holds only if $X$ is $\pP^{\dim(X)}$. For a Fano threefold $X$, this means $i(X)$ is either $1,2,3$ or $4$. If $i(X)=3$, then $X$ is a smooth quadric in $\pP^4$. If $i(X)=1$ or $2$, the most interesting cases are those with $\Pic(X)\cong \Z$ and they were classified by Iskovskih\cite{IP}.\ 

In this paper, we will only be interested in Fano threefolds $X$ with Picard rank $1$ and index $2$. Let $H$ be the ample generator of $\Pic(X)$. Let $d:=H^3$ denote the degree of $X$. Then Iskovskih’s classification\cite{IP} asserts $1\leqslant d\leqslant 5$, and:
\begin{itemize}
    \item if $d=5$, $X_5\hookrightarrow\pP^6$ is a codimension $3$ linear section of $\mathrm{Gr}(2,5)$ in its Pl\"ucker embedding;
    \item if $d=4$, $X_4\hookrightarrow\pP^5$ is a complete intersection of two smooth quadrics;
    \item if $d=3$, $X_3\hookrightarrow\pP^4$ is a cubic threefold;
    \item if $d=2$, $X_2\to\pP^3$ is a double cover of $\pP^3$ ramified in a quartic surface;
    \item if $d=1$, $X_1$ is a hypersurface of degree $6$ in a weighted projective space $\pP(1,1,1,2,3)$.
\end{itemize}
Let $X$ be a Fano threefold of Picard rank $1$ and index $2$. Then 
\begin{align*}
    H^2(X,\Z)=H^4(X,\Z)=H^6(X,\Z)=\Z
\end{align*}
and they are generated by the class of a hyperplane, a line and a point respectively. As a result, we will refer to the Chern classes of coherent sheaves on $X$ as integers. The ample generator of the Picard group will be denoted by $\cO_X(1)$, thus $\omega_X\cong\cO_X(-2)$. It is an elementary computation to check that 
\begin{align*}
    \mathrm{td}(\mathcal{T}_X)=(1,1,1+\frac{d}{3},1).
\end{align*}
\subsection{Derived categories of of Fano threefolds of index $2$}
Let $X$ be a Fano threefold of Picard rank $1$ and index $2$. In this section, we review some facts about the bounded derived category of coherent sheaves on $X$. For basic notions on derived category and semiorthogonal decomposition, we refer the readers to \cite{H}.
\begin{defn}\cite{Ku1}
    The collection of line bundles $\{\cO_X,\cO_X(1)\}$ is exceptional. The Kuznetsov component $\mathsf{Ku}(X)$ is defined by the semiorthogonal decomposition
    \begin{align*}
        \mathrm{D^b}(X)=\langle\mathsf{Ku}(X),\cO_X,\cO_X(1)\rangle.
    \end{align*}
\end{defn}
We describe $\mathsf{Ku}(X)$ for $2\leqslant d\leqslant 5$:
\begin{itemize}
    \item if $d=5$, $\Ku(X_5)\cong \mathrm{D^b}(Q_3)$ where $Q_3$ is the Kronecker quiver with three arrows;
    \item if $d=4$, $\Ku(X_4)\cong \mathrm{D^b}(C)$ where $C$ is a smooth projective curve of genus $2$;
    \item if $d=3$, $\Ku(X_3)$ has a Serre functor $S_{\Ku(X_3)}$ satisfying $S^3_{\Ku(X_3)}=[5]$;
    \item if $d=2$, $\Ku(X_2)$ has a Serre functor $S_{\Ku(X_2)}$ satisfying $S_{\Ku(X_3)}=\iota[2]$ where $\iota$ is the involution induced by the double cover.
\end{itemize}
\subsection{Instanton bundles}
Let $X$ be a Fano threefold of Picard rank $1$, index $2$ and degree $d$.
\begin{defn}\cite{Ku2}
    An \emph{instanton bundle of charge $n$} on $X$ is a stable vector bundle $E$ of rank $2$ with $c_1(E)=0, c_2(E)=n$, enjoying the instantonic vanishing condition: 
    \begin{align*}
        H^1(X,E(-1))=0.
    \end{align*}
\end{defn}
We mention that the charge $c_2(E)\geqslant 2$ \cite[Corollary 3.2]{Ku2}. Instanton bundles of charge $2$ are called the \emph{minimal instantons}. We also mention here that if $E$ is a minimal instanton, then $E\in \mathsf{Ku}(X)$ by \cite[Lemma 3.1]{Ku2}.\

By definition, the moduli space of minimal instanton bundles is an open subscheme of the moduli space of Gieseker-semistable sheaves with Chern character $(2,0,-2,0)$. The classification of Gieseker-semistable sheaves with Chern character $(2,0,-2,0)$ was first achieved by Druel\cite{D} for $d=3$, and later generalized to $d=5,4$ by the author\cite{Q1}\cite{Q2}. 
\begin{prop}\label{classfication}
    Assume $3\leqslant d\leqslant 5$. Let $E$ be a Gieseker-semistable sheaf on $X$ with Chern character $(2,0,-2,0)$. If $E$ is stable, then either $E$ is locally free or $E$ is associated to a smooth conic $Y\subset X$ so that we have an exact sequence:
     \begin{align*}
         0\to E\to H^0(\theta(1))\otimes \cO_X\to \theta(1)\to 0
     \end{align*}
     where $\theta$ is the theta-characteristic of $Y$.\
     
     If $E$ is strictly Gieseker-semistable, then $E$ is the extension of two ideal sheaves of lines.
\end{prop}
\begin{rem}\label{deg5vanishing}
    Along with the above classification, it was proved in \cite{D}\cite{Q2} that for $d=3$ or $4$, any Gieseker-semistable sheaf $E$ with Chern character $(2,0,-2,0)$ satisfies the vanishing condition:
    \begin{align}\label{inst van}
        H^1(E(-1))=0.
    \end{align}
    In \cite{Q1}, the author conjectured that the same thing holds for $d=5$. We also mention that for a Gieseker-semistable sheaf $E$ with Chern character $(2,0,-2,0)$, satisfying (\ref{inst van}) is equivalent to $E\in \mathsf{Ku}(X)$.
\end{rem}
Using this description, the moduli space of Gieseker-semistable sheaves with Chern character $(2,0,-2,0)$ were studied in \cite{D}\cite{Q1}\cite{Q2}.
\begin{thm}\label{moduli}
    Assume $3\leqslant d\leqslant 5$. The moduli space $M_d$ of Gieseker-semistable sheaves with Chern character $(2,0,-2,0)$ and satisfying the vanishing condition (\ref{inst van}) is a smooth projective variety of dimension $5$. More specifically:
    \begin{itemize}
        \item if $d=5$, $M_5$ is isomorphic to $\pP^5$; 
        \item if $d=4$, $M_4$ is a $\pP^3$-bundle over the Jacobian of the genus $2$ curve $C$ mentioned in Section 2.2;
        \item if $d=3$, $M_3$ is isomorphic to the blow-up
of the intermediate Jacobian in (minus) the Fano surface of lines.
    \end{itemize}
\end{thm}
\section{Review on stability conditions}
\subsection{Slope-stability and Gieseker-stability}
Let $X$ be a smooth projective threefold. Fix an ample line bundle $H$ on $X$. When $X$ is a Fano threefold of Picard rank one, we will always take the ample generator of the Picard group as $H$.
\begin{defn}
    For any coherent sheaf $F$ on $X$, the \emph{(Mumford-Takemoto)-slope} is defined as
    \begin{align*}
        \mu_H(F):=
        \begin{cases}
            \frac{H^2\ch_1(F)}{H^3\ch_0(F)} &\mbox{ if $\ch_0(F)\neq 0$;}\\
            +\infty &\mbox{ if $\ch_0(F)= 0$.}
        \end{cases}
    \end{align*}
    A coherent sheaf $F$ is slope-(semi)stable if for any non-trivial proper subsheaf $F'\subset F$, we have $\mu_H(F')<(\leqslant )\mu_H(F/F').$
\end{defn}

Next we recall the notion of Gieseker-stability. We will also use the refined notion of 2-Gieseker-stability introduced in \cite[Section 4]{BBF et} and \cite[Section 2.5]{JM}. We follow the exposition of \cite{BBF et}. 
\begin{defn}
    We define a pre-order on the polynomial ring $\R[m]$ as follows:
    \begin{enumerate}
        \item For all non-zero $f\in\R[m]$, we have $f\prec0$.
        \item If $\deg(f)>\deg(g)$ for non-zero $f,g\in\R[m]$, then $f\prec g$.
        \item If $\deg(f)=\deg(g)$ for non-zero $f,g\in\R[m]$ and let $l_f$ and $l_g$ be the leading coefficient of $f$ and $g$, then $f\preceq g$ if and only if $f(m)/l_f\leqslant g(m)/l_g$ for $m\gg 0$.
    \end{enumerate}
\end{defn}
For any $F\in\Coh(X)$, we denote its \emph{Hilbert polynomial} by $P(F):=\chi(F(mH))=\sum_{i=0}^3a_im^i$. Moreover, let $P_2(F):=\sum_{i=1}^3a_im^i$.
\begin{defn}
    \begin{enumerate}
        \item We say $F$ is \emph{Gieseker-(semi)stable} if for all nontrivial proper subsheaf $F'\subset F$, the inequality $P(F')\prec(\preceq)P(F)$ holds.
        \item We say $F$ is \emph{2-Gieseker-(semi)stable} if for all nontrivial proper subsheaf $F'\subset F$, the inequality $P_2(F')\prec(\preceq)P_2(F/F')$ holds.
    \end{enumerate} 
\end{defn}
The three notions imply each other in the following way:
\begin{align*}
    \mbox{slope-stable}\Rightarrow \mbox{2-Gieseker-stable}\Rightarrow\mbox{Gieseker-stable}&\\
    \Downarrow&\\
    \mbox{slope-semistable}\Leftarrow \mbox{2-Gieseker-semistable}\Leftarrow\mbox{Gieseker-semistable}&
\end{align*}

\subsection{(Weak) stability conditions on triangulated categories}
In this section we review elements in the theory of (weak) stability conditions on threefolds which are essential for our discussion. We refer the readers to \cite{Br07} for basic notions of t-structure and slicing.\ 

Let $\mathcal{D}$ be a triangulated category. We first recall the notion of a heart.
\begin{defn}
    A \emph{heart of a bounded t-structure} on $\mathcal{D}$ is a full additive subcategory $\mathcal{A}$ such that:
    \begin{enumerate}
        \item if $i>j$ are integers, then $\Hom_{\mathcal{D}}(A[i],B[j])=0$ for all $A,B\in \mathcal{A}$.
        \item for any nonzero object $F\in\mathcal{D}$, there exists a sequence of morphisms
        \begin{align*}
         0=F_0\xrightarrow{\phi_1}F_1\xrightarrow{\phi_2}\cdots\xrightarrow{\phi_m}F_m=F
        \end{align*}
        so that $\mathrm{Cone}(\phi_i)$ is of the form $A_i[k_i]$ for $A_i\in \mathcal{A}$ and integers $k_1>k_2>\cdots k_m$.
    \end{enumerate}
\end{defn}
Note a heart $\mathcal{A}$ of a t-struecture is an abelian category.
\begin{defn}
    Let $\mathcal{A}$ be an abelian category. A group homomorphism $Z:K(\mathcal{A})\to \C$ is called a \emph{weak stability function} if for any nonzero object $F\in \mathcal{A}$, $\Im Z(F)\geqslant 0$ and $\Im Z(F)=0$ only if $\Re Z(F)\leqslant 0$.\ 
    
    We call $Z$ a \emph{stability function} if in addition, $\Im Z(F)=0$ implies $\Re Z(F)<0$ for $F\neq 0$.
\end{defn}
Fix a finite rank lattice $\Lambda$ and surjective group homomorphism $v:K(\mathcal{A})\to \Lambda$.
\begin{defn}
    A \emph{weak stability condition} on $\mathcal{D}$ with respect to $\Lambda$ is a pair $\sigma=(\mathcal{A},Z)$ where $\mathcal{A}$ is the heart of a bounded t-structure on $\mathcal{D}$ and $Z:\Lambda \to \C$ is a group homomorphism, such that the following conditions hold:
    \begin{enumerate}
        \item The composition $\mathcal{A}\xrightarrow{v}\Lambda\xrightarrow{Z}\C$ is a weak stability function. For $F\in \mathcal{A}$, we write $Z(F):=Z(v(F))$ for simplicity. We define
        \begin{align*}
            \mu_\sigma(F)=\begin{cases}
            -\frac{\Re Z(F)}{\Im Z(F)} &\mbox{if } \Im Z(F)>0;\\
            +\infty&\mbox{otherwise}
            \end{cases}
        \end{align*}
        We call $F$ \emph{$\sigma$-(semi)stable} if for all nonzero subobject $F'\subset F$, we have $$\mu_\sigma(F')<(\leqslant)\mu_\sigma(F/F').$$
        \item Any object of $\mathcal{A}$ has a Harder-Narasimhan(HN) filtration in $\sigma$-semistable objects (called HN factors).
        \item There exists a quadratic form $Q$ on $\Lambda\otimes \R$ such that $Q(F)\geqslant 0$ for any $\sigma$-semistable object $F\in \mathcal{A}$ and $Q$ is negative definite when restricted to the kernel of $Z$.
    \end{enumerate}
    If in addition $Z\circ v:K(\mathcal{A})\to \C$ is a stability function, we call $\sigma$ a \emph{Bridgeland stability condition}.\ 
    \end{defn}

    We use $\mathrm{Stab}(\mathcal{D})$ to denote the set of Bridgeland stability conditions on $\mathcal{D}$. Bridgeland\cite{Br07} showed that $\mathrm{Stab}(\mathcal{D})$ has the structure of a complex manifold. Moreover, if we use $\tilde{\mathrm{GL}}_2^+(\R)$ to denote the universal cover of $\mathrm{GL}_2^+(\R)$, then there is a right group action of $\tilde{\mathrm{GL}}_2^+(\R)$ on $\mathrm{Stab}(\mathcal{D})$. We refer the readers to \cite{Br07} for details of this action.\ 
    
    Given a weak stability condition $\sigma=(\mathcal{A},Z)$ on $\mathcal{D}$, one can construct a new heart of a bounded t-structure via the method of \emph{tilting}: let $\mu\in \R$, define
    \begin{align*}
        &\mathcal{T}^\mu_\sigma=\{E\in \mathcal{A}:\mbox{ all HN factors $F$ of $E$ have slope }\mu_\sigma(F)>\mu\};\\
        &\mathcal{F}^\mu_\sigma=\{E\in \mathcal{A}:\mbox{ all HN factors $F$ of $E$ have slope }\mu_\sigma(F)\leqslant \mu\}.
    \end{align*}

\begin{prop}\cite{HRS}\label{tilt}
The category
\begin{align*}
    \mathcal{A}^\mu_\sigma=<\mathcal{T}^\mu_\sigma, \mathcal{F}^\mu_\sigma[1]>
\end{align*}
is the heart of a bounded t-structure on $\mathcal{D}$. 
\end{prop}
Next we review the notion of tilt-stability. We will follow the exposition of \cite[Section 2]{BLMS}. For $j=0,\ldots, 3$, define $\Lambda^j_H\cong \Z^{j+1}$ as the lattice generated by vectors of the form
$$(H^3\ch_0(F),H^2\ch_1(F),\ldots,H^{3-j}\ch_j(F))\in \Q^{j+1}$$
together with the natural map $v^j_H:K(X)\to \Lambda^j_H$.\ 

Then the pair $(\Coh(X), Z_H)$ with
\begin{align*}
    Z_H(F)=-H^{2}\ch_1(F)+iH^3\ch_0(F)
\end{align*}
defines a weak stability condition with respect to $\Lambda_H^1$. In this case we can take the quadratic form $Q=0$ since $Z$ is injective. This notion of stability coincide with (Mumford-Takemoto)-slope-stability and will be referred to as such. We will use $\mu_H$ to denote the slope to $H$. Any slope-semistable sheaf $F$ satisfies the Bogomolov-Gieseker inequality:
\begin{align}\label{BG}
    \Delta_H(F):=(H^{2}\ch_1(F))^2-2H^3\ch_0(F)\cdot H\ch_2(F)\geqslant 0.
\end{align}
Choose a parameter $\beta\in\R$. By Proposition \ref{tilt},
\begin{defn}
    We denote by $\Coh^\beta_H(X)\subset \mathrm{D^b}(X)$ the heart of a bounded t-structure obtained by tilting the slope stability at $\mu_H=\beta$.
\end{defn}
\begin{rem}
    When $X$ is a Fano threefold of Picard rank $1$, we will always take the ample generator of the Picard group as $H$ and drop the subscript $H$ from related notations.
\end{rem}
For a coherent sheaf $F$, we consider the twisted Chern character $\ch^\beta(F)=e^{-\beta H}\ch(F)$. More explicitly:
\begin{align*}
    &\ch_0^\beta=\ch_0\\
    &\ch_1^\beta=\ch_1-\beta H\ch_0\\
        &\ch_2^\beta=\ch_2-\beta H\ch_1+\frac{\beta^2}{2}H^2\ch_0\\
    &\ch_3^\beta=\ch_3-\beta H\ch_2 +\frac{\beta^2}{2}H^2\ch_1-\frac{\beta^3}{6}H^3\ch_0.
    \end{align*}
\begin{prop}\cite[Proposition 2.12]{BLMS} Given $\alpha>0$, $\beta\in \R$, the pair $\sigma_{\alpha,\beta}=(\Coh^\beta(X),Z_{\alpha,\beta})$ with
\begin{align*}
    Z_{\alpha,\beta}=\half\alpha^2H^3\ch_0^\beta(F)-H\ch_2^\beta(F)+iH^{2}\ch_1^\beta(F)
\end{align*}
defines a weak stability condition on $\mathrm{D^b}(X)$ with respect to $\Lambda^2$. The quadratic form $Q$ can be given by the discriminant $\Delta_H$ defined in (\ref{BG}).\ 

These weak stability conditions vary continuously as $(\alpha,\beta)\in\R_{>0}\times\R$ varies. 
\end{prop}
When the choices of $(\alpha,\beta)$ are clear, $\sigma_{\alpha,\beta}$-(semi)stability is usually referred to as \emph{tilt-(semi)stability}. The notion of 2-Gieseker-stability occurs as limit of tilt stability:
\begin{prop}\cite[Proposition 4.8]{BBF et}\cite[Proposition 14.2]{Br08}\cite[Theorem 5.2]{JM}\label{2G}
Let $F\in \mathrm{D^b}(X)$ and $\beta<\mu(F)$. Then $F\in \Coh^\beta(X)$ and $F$ is $\sigma_{\alpha,\beta}$-(semi)stable for $\alpha\gg 0$ if and only if $F\in\Coh(X)$ and $F$ is 2-Gieseker-(semi)stable.
\end{prop}

We will also need to following result of \cite[Proposition 3.2]{Li}. The original proposition in \cite{Li} is more general, our version follows from it by explicit computation.
\begin{prop}\cite[Proposition 3.2]{Li}\label{Libound}
    Let $X$ be a Fano threefold of Picard rank $1$ and index $2$ and $F$ be $\sigma_{\alpha,\beta}$-stable with $\ch_0(F)\neq 0$ for some $\alpha>0, \beta\in \R$. 
    \begin{enumerate}
        \item If $d=5$ and $-\sqrt{\frac{3}{20}}\leqslant\mu_H(F)\leqslant \sqrt{\frac{3}{20}}$, then $\frac{H\ch_2(F)}{H^3\ch_0(F)}\leqslant0$.
        \item If $d=4$ and $\frac{\sqrt{3}}{4}\leqslant|\mu_H(F)|\leqslant1-\frac{\sqrt{3}}{4}$, then $\frac{H\ch_2(F)}{H^3\ch_0(F)}\leqslant\half(\mu_H(F))^2-\frac{3}{32}$.
        \item If $d=3$ and $-\half\leqslant\mu_H(F)\leqslant\half$, then $\frac{H\ch_2(F)}{H^3\ch_0(F)}\leqslant0$.
        \item If $d=3$ and $\half<|\mu_H(F)|\leqslant1$, then $\frac{H\ch_2(F)}{H^3\ch_0(F)}\leqslant|\mu_H(F)|-\half$.
        \item If $d=2$ and $-\half\leqslant\mu_H(F)\leqslant\half$, then $\frac{H\ch_2(F)}{H^3\ch_0(F)}\leqslant0$.
    \end{enumerate}
    Moreover, if the equality holds in any of the cases above, then $\ch_0(F)$ is $1$ or $2$.
\end{prop}

Finally we recall a variant of the tilt stability conditions, which is needed in the next section. Fix $\mu\in \R$, apply Proposition \ref{tilt} to $\sigma_{\alpha,\beta}$, we obtain a heart, which we denote by
\begin{align*}
    \Coh^\mu_{\alpha,\beta}(X):=\mathcal{A}^\mu_{\sigma_{\alpha,\beta}}.
\end{align*}
Let $u\in\C$ be the unit vector in the upper half plane with $\mu=-\frac{\Re u}{\Im u}$.
\begin{prop}\cite[Proposition 2.15]{BLMS}\label{2nd tilt}
The pair $\sigma^\mu_{\alpha,\beta}:=(\Coh^\mu_{\alpha,\beta}(X), Z^\mu_{\alpha,\beta})$, where 
\begin{align*}
    Z^\mu_{\alpha,\beta}:=\frac{1}{u}Z_{\alpha,\beta}
\end{align*}
is a weak stability condition on $\mathrm{D^b}(X)$ with respect to $\Lambda^2$.

\end{prop}

\subsection{Stability conditions on the Kuznetsov component} Let $X$ be a Fano threefold of Picard rank $1$ and index $2$. We will induce stability conditions on $\mathsf{Ku}(X)$ from the weak stability conditions $\sigma^0_{\alpha,\beta}$ of Proposition \ref{2nd tilt}. Set $\mathcal{A}(\alpha,\beta):=\Coh^0_{\alpha,\beta}(X)\cap \mathsf{Ku}(X)$ and $Z(\alpha,\beta)=Z^0_{\alpha,\beta}|_{\mathsf{Ku}(X)}$. 
\begin{thm}\cite[Theorem 6.8]{BLMS}\cite[Theorem 3.3]{PY}\label{Kustab}
    Suppose $-\half\leqslant\beta<0, 0<\alpha<-\beta,$ or $-1<\beta<-\half, 0<\alpha\leqslant 1+\beta$. Then the pair
    \begin{align*}
        \sigma(\alpha,\beta):=(\mathcal{A}(\alpha,\beta),Z(\alpha,\beta))
    \end{align*}
    is a Bridgeland stability condition on $\mathsf{Ku}(X)$ with respect to $\Lambda^2_{\mathsf{Ku}(X)}$, where 
    \begin{align*}
        \Lambda^2_{\mathsf{Ku}(X)}:=\mbox{Im}(K(\mathsf{Ku}(X))\to K(X)\to \Lambda^2)\cong \Z^2.
    \end{align*}
\end{thm}
We mention the existence of such a Bridgeland stability condition on $\Ku(X)$ was first proved in \cite[Theorem 6.8]{BLMS} with $\beta=-\half$. The specific triangular region was computed in \cite[Theorem 3.3]{PY}. From now on, we use 
\begin{align*}
    V:=\{(\alpha,\beta)\in\R_{>0}\times\R:-\half\leqslant\beta<0, \alpha<-\beta,\text{ or }-1<\beta<-\half, \alpha\leqslant 1+\beta\}
\end{align*}
to denote this triangular region.

In fact, the stability conditions constructed above are the same up to the action of $\tilde{\mathrm{GL}}_2^+(\R)$.
\begin{prop}\cite[Proposition 3.6]{PY}
Fix $0<\alpha_0<\half$. For any $(\alpha,\beta)\in V$, there is $\tilde{g}\in\tilde{\mathrm{GL}}_2^+(\R)$ such that $\sigma(\alpha,\beta)=\sigma(\alpha_0,-\half)\cdot\tilde{g}$. 
\end{prop}
In light of this result, we set
\begin{align*}
    \mathcal{K}:=\sigma(\alpha_0,-\half)\cdot\tilde{\mathrm{GL}}_2^+(\R)\subset\mathrm{Stab}(\Ku(X))
\end{align*}
Then all stability conditions constructed in Theorem \ref{Kustab} are in $\mathcal{K}$.

\section{Bridgeland stability of minimal instanton bundles}
For the rest of this paper, let $X$ be a Fano threefold of Picard rank one, index two and degree $d$. Let $E$ be a minimal instanton bundle on $X$.
The Chern character of $E$ is 
\begin{align*}
    \mathrm{ch}(E)=(2,0,-2,0),
\end{align*}
and the twisted Chern character with respect to $-\half$ till degree $2$ is 
\begin{align*}
    \mathrm{ch}^{-\half}_{\leqslant2}(E)=(2,H,\frac{d-8}{4}).
\end{align*}

\begin{prop}\label{stab}
Let $E$ be a minimal instanton bundle on $X$. Then $E$ is $\sigma(\alpha_0,-\half)$-semistable for some $0<\alpha_0<\half$. As a result, $E$ is semistable for any stability condition in $\mathcal{K}$.
\end{prop}
\begin{proof}
$E$ is stable by definition. By \cite[Lemma 1.23]{Sa}, $E$ is $\mu$-stable.  As $\mu_H(E)=0>-\half$, we have $E\in \Coh^{-\half}(X)$. Since $H^2\mathrm{ch}_1^{-\half}(E)>0$, $E$ is $\sigma_{\alpha,-\half}$-stable for $\alpha\gg 0$ by \cite[Lemma 2.7(c)]{BMS}\

Next we show that $E$ is $\sigma_{\alpha,-\half}$-semistable for some $\alpha<\half$. A wall would be given by a short exact sequence in the heart $\Coh^{-\half}(X)$ of the form
\begin{align*}
    0\to E'\to E\to E''\to 0,
\end{align*}
such that the following conditions hold:
\begin{enumerate}
    \item $\mu_{\alpha,-\half}(E')=\mu_{\alpha,-\half}(E)=\mu_{\alpha,-\half}(E'')$;
    \item $\Delta_H(E')\geqslant 0$, $\Delta_H(E'')\geqslant 0$;
    \item $\Delta_H(E')\leqslant \Delta_H(E)$, $\Delta_H(E'')\leqslant \Delta_H(E)$;
    \item $
    \mathrm{ch}^{-\half}_{\leqslant2}(E')\neq(1,\half H,\frac{d-8}{8}), \mathrm{ch}^{-\half}_{\leqslant2}(E')\neq (2,H,\frac{d-8}{4}).$
\end{enumerate}
Note we require the last condition because $\ch(E)$ is not primitive and we are considering walls that might break the semistability of $E$.\ 

The truncated twisted characters of $E'$ and $E''$ have to satisfy
\begin{align*}
    (2,H,\frac{d-8}{4})=(a,\frac{b}{2}H,\frac{c}{8})+(2-a,\frac{2-b}{2}H,\frac{2d-16-c}{8})
\end{align*}
for some $a,b,c\in \Z$. Since $E'$ and $E''$ are in $\Coh^{-\half}(X)$, we have $b\geqslant 0$ and $2-b\geqslant 0$.
Thus $b=0,1$ or $2$.\ 

As $\mu_{\alpha,-\half}(E)=\frac{d-8-4d\alpha^2}{4d}$ and $\Delta_H(E)=8d$, the first three conditions become:
\begin{enumerate}
    \item $\frac{1}{b}(\frac{c}{4d}-\alpha^2a)=\frac{d-8-4d\alpha^2}{4d}=\frac{1}{2-b}(\frac{2d-16-c}{4d}-\alpha^2(2-a))$;
    \item $(\frac{b}{2})^2-\frac{ac}{4d}\geqslant 0$, $(\frac{2-b}{2})^2-\frac{(2-a)(2d-16-c)}{4d}\geqslant 0$;
    \item $(\frac{b}{2})^2-\frac{ac}{4d}\leqslant \frac{8}{d}$, $(\frac{2-b}{2})^2-\frac{(2-a)(2d-16-c)}{4d}\leqslant \frac{8}{d}$. 
\end{enumerate}
Suppose $b=0$. If $a\neq0$, then 
\begin{align*}
    4d\alpha^2=\frac{c}{a}>0 \text{ and } ac\leqslant 0
\end{align*}
which is impossible. If $a=0$, then $c=0$ by the first equation. This means either $E'$ or $E''$ has twisted Chern character $(2,H,\frac{d-8}{4})$. The case that $E'$ has twisted Chern character $(2,H,\frac{d-8}{4})$ is excluded by condition (4). If $E''$ has twisted Chern character $(2,H,\frac{d-8}{4})$, then $E$ would have a subobject with infinite slope, again contradicting the stability of $E$ for $\alpha\gg 0$. The case $b=2$ can be excluded by symmetry and the above argument.\ 

Suppose $b=1$, then $2-b=1$. By condition (4), we can assume $a\neq 1$, then either $a$ or $2-a$ will be greater than or equal to $2$. We assume $a\geqslant 2$ without loss of generality. We have then either $\ch(E')$ or $\ch(E'')$ being equal to
\begin{align*}
    (a,\half (1-a)H,.....).
\end{align*}
Since $\half (1-a)\in\Z$, $a$ must be odd and thus $a\geqslant 3$. Using (1) and the assumption $b=1$, we obtain
\begin{align*}
    c=d-8+4d\alpha^2(a-1).
\end{align*}
Combine this with the second equation in (2), we obtain
\begin{align*}
    \frac{(a-2)(4d\alpha^2(a-1)-d+8)}{4d}&\leqslant \frac{1}{4}
\end{align*}
which simplifies to 
\begin{align*}
        4d\alpha^2(a-1)&\leqslant \frac{d}{a-2}+d-8\leqslant 2d-8.
\end{align*}
This leads to contradictions when $1\leqslant d\leqslant 4$, since the left hand side is positive. In these cases there are no walls for the tilt-semistability of $E$.\

If $d=5$, the above equation becomes
\begin{align*}
    \alpha^2\leqslant \frac{1}{10(a-1)}\leqslant 1/20.
\end{align*}
In this case there are no walls for $\alpha>\frac{1}{\sqrt{20}}.$\ 

As a result, for all $1\leqslant d\leqslant 5$, we can choose $\frac{1}{\sqrt{20}}<\alpha_0<\half$ so that $E$ is $\sigma_{\alpha_0,-\half}$-semistable. Since $E$ is torsion free, this implies $E$ is $\sigma^0_{\alpha_0,-\half}$-semistable and thus $\sigma(\alpha_0,-\half)$-semistable for the same $0<\alpha_0<\half$.
By \cite[Section 3.3]{PY}, $E$ is semistable for all stability conditions in $\mathcal{K}\subset \mathrm{Stab}(\mathsf{Ku}(X))$.
\end{proof}


\begin{cor}\label{stabcor}
Let $E$ be a minimal instanton bundle on $X$. If the degree of $X$ satisfy $d\geqslant2$, then $E$ is stable for any stability condition in $\mathcal{K}$.
\end{cor}
\begin{proof}
It suffices to show the stability of $E$ for $\sigma(\alpha_0,-\half)$. Suppose $E$ is strictly $\sigma(\alpha_0,-\half)$-semistable. By \cite[Lemma 3.9]{PY}, it has two Jordan-H\"older factors, each having numerical class as that of the ideal sheaf of a line. By \cite[Proposition 4.6]{PY}, the Jordan-H\"older factors are ideal sheaves of lines. This contradicts the $\mu$-stability of $E$. Thus $E$ is $\sigma(\alpha_0,-\half)$-stable.
\end{proof}

For the rest of this section, assume $d\geqslant3$. We look at non-locally free generalizations of minimal instanton bundles. Let $C$ be a smooth conic on $X$ with $d\geqslant 3$. Note $C\simeq \pP^1$ and we denote by $\theta$ its theta characteristic (dual of the ample generator in $\mathrm{Pic}(C)$). Let $E$ be the torsion free sheaf defined by the short exact sequence
\begin{align}\label{conic}
    0\to E\to H^0(\theta(1))\otimes \cO_X\to \theta(1)\to 0.
\end{align}
Then by \cite{D}\cite{Q1}]\cite{Q2}, $E$ is Gieseker-stable and $E\in \mathsf{Ku}(X)$.
\begin{prop}\label{conicstab}
Let $E$ be a sheaf on $X$ as defined in (\ref{conic}). Then $E$ is stable for any stability condition in $\mathcal{K}$.
\end{prop}
\begin{proof}
The proof is essentially the same as above. Although $E$ is not slope-stable, we note since $E$ is Gieseker-stable, it is 2-Gieseker-semistable. By Proposition \ref{2G}, $E$ is $\sigma_{\alpha,-\half}$-semistable for $\alpha\gg 0$. We can then apply the arguments for walls in Proposition \ref{stab} to show $E$ is semistable for stability conditions in $\mathcal{K}$. The fact that $E$ is stable follows from the arguments of Corollary \ref{stabcor}.
\end{proof}
\begin{rem}\label{semistab}
    By \cite[Proposition 4.4]{PY}, extensions of ideal sheaves of lines on $X$ are strictly $\sigma$-semistable for every $\sigma\in \mathcal{K}$. In fact, the proof of Corollary \ref{stabcor} shows that they are the only strictly $\sigma$-semistable objects with numerical class $2[\cI_l]$.
\end{rem}

\section{Description of the moduli space $M_\sigma(\mathsf{Ku}(X),2[\cI_l])$}
In this section we will show that for $\sigma\in\mathcal{K}$, $\sigma$-semistable objects in $\Ku(X)$ with numerical class $2[\cI_l]\in\mathcal{N}(\mathsf{Ku}(X))$ are Gieseker-semistable sheaves with Chern character $(2,0,-2,0)$ up to shifts when $d\geqslant 3$. We will then use this result to describe $M_\sigma(\mathsf{Ku}(X),2[\cI_l])$. We need the following two lemmas.

\begin{lem}\label{2 Gieseker to Gieseker}
Let $E$ be a 2-Gieseker-semistable sheaf with $rk(E)=2, c_1(E)=0, c_2(E)=2$ and $c_3(E)=0$ on $X$. Then $E$ is Gieseker-semistable.
\end{lem}
\begin{proof}
Suppose $E$ is Gieseker-unstable. Then the Harder-Narasimhan filtration of $E$ in $\Coh(X)$ is of the form
\begin{align*}
    0\to E'\to E\to E''\to 0
\end{align*}
with $E', E''$ Gieseker-semistable of rank $1$ and $p(E')\succ p(E)$. Note $\ch(E)=(2,0,-2,0)$. Since $E$ is 2-Gieseker-semistable, it is easy to see that we must have $\ch_1(E')=\ch_1(E'')=0$ and $\ch_2(E')=\ch_2(E'')=-1$. Then $E'$ and $E''$ are ideal sheaves of closed subschemes of dimension $1$, which we denote by $D'$ and $D''$ respectively.  Since $E'$ is Gieseker-destabilizing, $\ch(E')=(1,0,-1,m)$ with $m\geqslant 1$. The Hilbert polynomial of $D'$ is thus $P_{D'}(t)=t+1-m$, this contradicts \cite[Corollary 1.38(1)]{Sa}.\ 
\end{proof}

\begin{lem}\label{nowall-1}
Assume $d\geqslant 3$. Let $E\in \cO_X^\perp$ be an object with Chern character $(2,0,-2,0)$. Then the vertical line $\beta=-1$ does not intersect any actual wall for the $\sigma^0_{\alpha,\beta}$-stability of $E$.
\end{lem}
\begin{proof}
    The proof is similar to that of Proposition \ref{stab}. We provide it for the convenience of the readers. The twisted Chern character of $E$ is 
    \begin{align*}
        \ch^{-1}_{\leqslant 2}(E)=(2,2H,d-2)
    \end{align*}
    A wall will provide a equation of the twisted Chern characters:
    \begin{align}\label{split}
        (2,2H,d-2)=(a,bH,\frac{c}{2})+(2-a,(2-b)H,d-2-\frac{c}{2})
    \end{align}
    with $a,b,c\in\Z$ satisfying
    \begin{enumerate}
    \item $\frac{1}{b}(\frac{c}{2}-\half d\alpha^2a)=\frac{d-2-d\alpha^2}{2}=\frac{1}{2-b}(d-2-\frac{c}{2}-\half d\alpha^2(2-a))$;
    \item $b^2-\frac{ac}{d}\geqslant 0$, $(2-b)^2-\frac{(2-a)(2d-4-c)}{d}\geqslant 0$
    \item $b^2-\frac{ac}{d}\leqslant \frac{8}{d}$, $(2-b)^2-\frac{(2-a)(2d-4-c)}{d}\leqslant \frac{8}{d}$. 
\end{enumerate}
By the definition of a stability function on an abelian category, $b$ and $2-b$ cannot have opposite signs, so $b(2-b)\geqslant 0$ and thus $0\leqslant b\leqslant 2$.\ 

Suppose $b=0$. If $a\neq 0$, then 
\begin{align*}
    \half d\alpha^2=\frac{c}{2a}>0 \text{  and  } \frac{ac}{d}\leqslant 0
\end{align*}
which is impossible. If $a=0$, then $c=0$. If the twisted Chern character $(a,bH,\frac{c}{2})$ corresponds to a subobject, then $d-2-d\alpha_0^2=0$. So $\alpha_0^2=1-\frac{2}{d}$. In this case we have a short exact sequence in $\Coh^0_{\alpha_0,-1}$
\begin{align*}
    0\to\cO_Z\to E[s]\to B\to 0
\end{align*}
where $s\in\Z$ is a potential shift and $Z$ is $0$-dimensional subscheme on $X$. One can easily check $\cO_X\in\Coh^0_{(\alpha_0,-1)}$. By the assumption that $E\in\cO_X^\perp$, we have $\Hom(\cO_X,\cO_Z)=\Hom(\cO_X,B[-1])=0$, this is absurd if $Z$ is non-empty.
 On the other hand, having a quotient object with twisted Chern character $(0,0,0)$ does not affect semistability. The case $b=2$ follows by symmetry.\ 

Suppose $b=1$, we have $c=d\alpha^2(a-1)+d-2$. Without loss of generality, we assume $a\geqslant 1$. Note if $a=1$, the two terms on the right hand side of equation (\ref{split}) are equal, which means such a wall will not affect semistability. For $a\geqslant 2$, the first equation of (2) provides
\begin{align*}
    a(a-1)\alpha^2\leqslant 1-a(1-\frac{2}{d}).
\end{align*}
For $d\geqslant 4$, we get 
\begin{align*}
    a(a-1)\alpha^2\leqslant 1-2(1-\half)=0.
\end{align*}
which is absurd since the left hand side is positive.\

For $d=3$, we obtain similar contradiction as above if $a\geqslant 3$. If $a=2$, $\alpha^2\leqslant \frac{1}{6}$. In this case $c$ cannot be an integer.
\end{proof}

\begin{prop}\label{stabrev}
Assume $d\geqslant 3$. If $F\in \Ku(X)$ is $\sigma$-stable for some $\sigma\in \mathcal{K}$ with $[F]=2[\cI_l]\in\mathcal{N}(\mathsf{Ku}(X))$, then $F\cong E[2k]$ for some Gieseker-stable sheaf $E$ of rank $2, c_1(E)=0,c_2(E)=2, c_3(E)=0$ and $k\in \Z$. 
\end{prop}
\begin{proof}
We follow the idea of the proof of \cite[Proposition 4.6]{PY}. As $[F]=2[\cI_l]\in \mathcal{N}(\mathsf{Ku}(X))$ and $\chi(\cI_l,\cI_l)=-1$, the following conditions hold:
    \begin{align*}
        \chi(\cO_X,F)=0,\chi(\cO_X(1),F)=0,\chi(\cI_l,F)=-2,\chi(F,\cI_l)=-2.
    \end{align*}
By Hirzebruch-Riemann-Roch Theorem and a similar computation as \cite{PY}, we obtain $\ch(F)=2\ch(\cI_l)=(2,0,-2,0)$.\ 

Recall that $V$ is the triangular region parametrizing stability conditions on $\Ku(X)$ defined after Theorem 3.15. By \cite[Proposition 3.6]{PY} and the assumption, $F$ is $\sigma(\alpha,\beta)$-stable for every $(\alpha,\beta)\in V$. In particular, for all $(\alpha,\beta)\in V$ satisfying $\beta^2-\alpha^2\leqslant 2/d$, we have $F[2k+1]\in\mathcal{A}(\alpha,\beta)$ for some integer $k$. Up to shifting, we assume $G:=F[1]\in\mathcal{A}(\alpha,\beta)$ is $\sigma(\alpha,\beta)$-stable for $(\alpha,\beta)$ satisfying $\beta^2-\alpha^2\leqslant 2/d$. Then $G$ has slope
\begin{align*}
   \mu_{ \sigma^0_{\alpha,\beta}}(G)=\frac{-\beta}{\frac{1}{d}+\half\alpha^2-\half\beta^2}.
\end{align*}
In particular, $ \mu_{\sigma^0_{\alpha,\beta}}(G)=+\infty$ if 
\begin{align}\label{parabola}
    \beta^2-\alpha^2=2/d. 
\end{align}
Since $d\geqslant 3$, then there exists pairs $(\bar{\alpha},\bar{\beta})\in V$ so that (\ref{parabola}) holds. $G$ is $\sigma^0_{\bar{\alpha},\bar{\beta}}$ semistable as it has the largest slope ($+\infty$) in heart. By Lemma \ref{nowall-1}, $\beta=-1$ is not on a wall for the $\sigma^0_{\alpha,\beta}$-stability of $G$. Moreover, the semicircle $\mathcal{C}$ with center $(0,-\frac{d+2}{2d})$ and radius $\frac{d-2}{2d}$ gives a numerical wall for $G$, potentially realized by $\cO_X(-1)[2]\in\Coh^0_{\alpha,\beta}(X)$.\ 

Assume that $\mathcal{C}$ is not an actual wall for $G$. All other walls would be nested semicircles in $\mathcal{C}$. Thus we may choose $(\bar{\alpha},\bar{\beta})=(\frac{d-2}{2d},-\frac{d+2}{2d})\in\mathcal{C}$, so that $G$ is $\sigma^0_{\bar{\alpha},\bar{\beta}}$-semistable and remains so for $\bar{\beta}$ approaching $-\half$.\ 

By definition of $\Coh^0_{\alpha,-\half}(X)$, we have a triangle
\begin{align*}
    A[1]\to G\to B
\end{align*}
where $A$(resp. $B$) is in $\Coh^{-\half}(X)$ with $\sigma_{\alpha,-\half}$-semistable factors having slope $\mu_{\alpha,-\half}\leqslant 0$(resp. $\geqslant 0$). Since $G$ is $\sigma^0_{\alpha,-\half}$-semistable, $B$ is either $0$ or supported on points. In any case, $A[1]$ is also $\sigma^0_{\alpha,-\half}$-semistable, which implies $A$ is $\sigma_{\alpha,-\half}$-semistable. Note $\ch(A)=(2,0,-2,l)$, where $l\geqslant 0$ is the length of the support of $B$. The wall computation in Proposition \ref{stab} shows that $A$ is $\sigma_{\alpha,-\half}$-semistable for every $\alpha>0$ if $d=3,4$ and $\alpha>\frac{1}{\sqrt{20}}$ if $d=5$. By Proposition \ref{2G}, $A$ is a 2-Gieseker-semistable sheaf. When $d=5$, applying \cite[Conjecture 4.1]{BLMS},\cite{Li} for $\alpha=\frac{1}{\sqrt{20}}$ $\beta=-\half$, we obtain
\begin{align*}
    l&\leqslant \frac{4}{3}\alpha^2+\frac{1}{3}+\frac{8}{3d}\\
    &\leqslant \frac{4}{3}\cdot\frac{1}{20}+\frac{1}{3}+\frac{8}{15}=\frac{14}{15}
\end{align*}
Thus $l=0$ if $d=5$. 
When $d\neq 5$, applying \cite[Conjecture 4.1]{BLMS},\cite{Li} for $\alpha=0$ $\beta=-\half$, we obtain
\begin{align*}
    l&\leqslant \frac{1}{3}+\frac{8}{3d}
\end{align*}
for $d\geqslant 3$. Then $l$ is $0$ or $1$ for $3\leqslant d\leqslant 4$.  We claim $l=0$. Suppose otherwise, since $G\in \mathsf{Ku}(X)$, we have $H^i(A)=0$ for all $i$ except when $i=2$ and $H^2(A)=\C$. Let $q$ be a generic point on $X$ and consider the elementary modification $A'$ of $A$ defined by the short exact sequence
\begin{align}\label{elemod}
    0\to A'\to A\to \cO_q\to 0.
\end{align}
Then $A'$ is a 2-Gieseker-semistable sheaf with $\ch(A')=(2,0,-2,0)$. By Lemma \ref{2 Gieseker to Gieseker}, $A'$ is Gieseker-semistable. By \cite{D},\cite{Q2}, $A'\in \mathsf{Ku}(X)$. Along with (\ref{elemod}), we see $H^2(A)=0$, which contradicts our computation above. Hence $l=0$ for all $d=3,4,5$. Consequently, $G=A[1]$ and $F=A\in \mathsf{Ku}(X)$ is a 2-Gieseker-semistable sheaf. By Lemma \ref{2 Gieseker to Gieseker}, $F$ is a Gieseker-semistable sheaf with rank $2, c_1(E)=0,c_2(E)=2, c_3(E)=0$. $F$ is Gieseker-stable since otherwise $F$ would be an extension of two ideal sheaves of lines, making it strictly $\sigma$-semistable.\ 

Assume $\mathcal{C}$ defines an actual wall for $G$ and $G$ becomes unstable for $\beta\to -\half$. Set $(\bar{\alpha},\bar{\beta})=(\frac{d-2}{2d},-\frac{d+2}{2d})$. $G$ is then strictly $\sigma^0_{\bar{\alpha},\bar{\beta}}$-semistable and we have a short exact sequence
\begin{align*}
    0\to P\to G\to Q\to 0
\end{align*}
in $\Coh^0_{\bar{\alpha},\bar{\beta}}(X)$, where $P,Q$ are $\sigma^0_{\bar{\alpha},\bar{\beta}}$-semistable with infinite slope. It is a tedious exercise to find the potential cases using
\begin{itemize}
    \item $\Im(Z^0_{\bar{\alpha},\bar{\beta}}(P))=\Im(Z^0_{\bar{\alpha},\bar{\beta}}(Q))=0$;
    \item $\Re(Z^0_{\bar{\alpha},\bar{\beta}}(P))\leqslant0, \Re(Z^0_{\bar{\alpha},\bar{\beta}}(Q))\leqslant 0$;
    \item $\Delta(P)\geqslant 0$, $\Delta(Q)\geqslant0$
\end{itemize}
We list them here:
\begin{enumerate}
    \item $d=5, \ch_{\leqslant 2}(P)=(1,-1,\frac{5}{2}), \ch_{\leqslant 2}(Q)=(-3,1,-\half)$;
    \item $d=4, \ch_{\leqslant 2}(P)=(0,-1,3), \ch_{\leqslant 2}(Q)=(-2,1,-1))$;
    \item $d=4, \ch_{\leqslant 2}(P)=(1,-1,2), \ch_{\leqslant 2}(Q)=(-3,1,0)$;
    \item $d=4, \ch_{\leqslant 2}(P)=(2,-2,4), \ch_{\leqslant 2}(Q)=(-4,2,-2)$;
    \item $d=3, \ch_{\leqslant 2}(P)=(0,-1,\frac{5}{2}), \ch_{\leqslant 2}(Q)=(-2,1,-\half)$;
    \item $d=3, \ch_{\leqslant 2}(P)=(1,-2,4), \ch_{\leqslant 2}(Q)=(-3,2,-2)$;
    \item $d=3, \ch_{\leqslant 2}(P)=(2,-3,\frac{11}{2}), \ch_{\leqslant 2}(Q)=(-4,3,-\frac{7}{2})$;
    \item $d=3, \ch_{\leqslant 2}(P)=(1,-1,\frac{3}{2}), \ch_{\leqslant 2}(Q)=(-3,1,\half)$;
    \item $d=3, \ch_{\leqslant 2}(P)=(2,-2,3), \ch_{\leqslant 2}(Q)=(-4,2,-1)$;
    \item $d=3, \ch_{\leqslant 2}(P)=(3,-3,\frac{9}{2}), \ch_{\leqslant 2}(Q)=(-5,3,-\frac{5}{2})$;
    \item $d=3, \ch_{\leqslant 2}(P)=(4,-4,6), \ch_{\leqslant 2}(Q)=(-6,4,-4)$.
\end{enumerate}
All the cases are ruled out by Lemmas \ref{first} to \ref{last}. This completes the proof.
\end{proof}

\begin{lem}\label{first}
Cases (2)(6) can be ruled out. 
\end{lem}
\begin{proof}
Note in both cases $\Delta(Q)=0$ and $\gcd(\ch_0(Q),\ch_1(Q))=0$. Let $\cH^i(Q)$ denote the $i$-th cohomology of $Q$ in $\Coh^{\bar{\beta}}(X)$. Since $Q$ is $\sigma^0_{\bar{\alpha},\bar{\beta}}$-semistable, we see  $\cH^0(Q)$ is either $0$ or supported on points. Moreover, $\cH^{-1}(Q)$ is $\sigma_{\bar{\alpha},\bar{\beta}}$-semistable with $\mu_{\bar{\alpha},\bar{\beta}}(\cH^{-1}(Q))=0$. By \cite[Corollary 3.10]{BMS}, $\cH^{-1}(Q)$ is in fact $\sigma_{\bar{\alpha},\bar{\beta}}$-stable. Then cases (2)(6) are ruled out by Proposition \ref{Libound}(2) and (4) respectively.
\end{proof}

\begin{lem}
Cases (5)(7) can be ruled out.
\end{lem}
\begin{proof}
Arguing as in the previous lemma, we get $\cH^{-1}(Q)$ is $\sigma_{\bar{\alpha},\bar{\beta}}$-semistable with $\mu_{\bar{\alpha},\bar{\beta}}(\cH^{-1}(Q))=0$. If $\cH^{-1}(Q)$ is in fact $\sigma_{\bar{\alpha},\bar{\beta}}$-stable, then cases (5)(7) can be ruled out by Proposition \ref{Libound}(3) and (4) respectively.

If $\cH^{-1}(Q)$ is strictly $\sigma_{\bar{\alpha},\bar{\beta}}$-semistable. 
Note for any $F\in \Coh^{\bar{\beta}}(X)$, $\Delta(F)/d$ is an integer. In both cases (5)(7), we have $\Delta(\cH^{-1}(Q))/3=1$. By \cite[Lemma 3.9]{BMS} \cite[Corollary 3.10]{BMS}, there is a short exact sequence
\begin{align*}
    0\to Q'\to \cH^{-1}(Q)\to Q''\to 0
\end{align*}
in $\Coh^{\bar{\beta}}$ where $Q'$ and $Q''$ are $\sigma_{\bar{\alpha},\bar{\beta}}$-semistable, with $\mu_{\bar{\alpha},\bar{\beta}}(Q')=\mu_{\bar{\alpha},\bar{\beta}}(Q'')=0$ and $\Delta(Q')=\Delta(Q'')=0$. It is a straightforward computation to check that we will have either $Q'$ or $Q''$ with $\ch_{\leqslant2}=(3,-2,2)$. We can then draw a contradiction using the arguments in the previous lemma for case (6).
\end{proof}

\begin{lem}
Case (1) can be ruled out.
\end{lem}
\begin{proof}
    As before, $\cH^{-1}(Q)$ is $\sigma_{\bar{\alpha},\bar{\beta}}$-semistable with $\mu_{\bar{\alpha},\bar{\beta}}(\cH^{-1}(Q))=0$. If $\cH^{-1}(Q)$ is in fact $\sigma_{\bar{\alpha},\bar{\beta}}$-stable, then cases (1) can be ruled out by Proposition \ref{Libound}(1).\ 
    
    If $\cH^{-1}(Q)$ is strictly $\sigma_{\bar{\alpha},\bar{\beta}}$-semistable. We have $\Delta(\cH^{-1}(Q))/5=2$. By \cite[Lemma 3.9]{BMS} \cite[Corollary 3.10]{BMS}, there is a short exact sequence
    \begin{align*}
        0\to Q'\to \cH^{-1}(Q)\to Q''\to 0
    \end{align*}
    in $\Coh^{\bar{\beta}}$ where $Q'$ and $Q''$ are $\sigma_{\bar{\alpha},\bar{\beta}}$-semistable, with $\mu_{\bar{\alpha},\bar{\beta}}(Q')=\mu_{\bar{\alpha},\bar{\beta}}(Q'')=0$ and $\Delta(Q')$,$\Delta(Q'')$ either $0$ or $5$. We denote $\ch_{\leqslant2}(Q')=(a,bH,\frac{c}{2})$, where $a,b,c\in\Z$. Then $\ch_{\leqslant2}(Q'')=(3-a,(-1-b)H,\frac{1-c}{2})$
    \begin{align*}
        \mu_{\bar{\alpha},\bar{\beta}}(Q')=0 &\Longleftrightarrow 2a+7b+c=0,\\
        \Delta(Q')/5=0\mbox{ or }1&\Longleftrightarrow 5b^2-ac=0\mbox{ or }1,\\
        \Delta(Q'')/5=0\mbox{ or }1&\Longleftrightarrow 5(1+b)^2-(3-a)(1-c)=0\mbox{ or }1.
    \end{align*}
    It is then straightforward but tedious to check the above equations have no integer solutions in $(a,b,c)$.
\end{proof}

\begin{lem}\label{last}
Cases (3)(4)(8)(9)(10)(11) can be ruled out.
\end{lem}

\begin{proof}
In these cases $\ch_{\leqslant 2}(P)=n\cdot(1, -H,\frac{H^2}{2})$. Let $\cH^i(P)$ denote the $i$-th cohomology of $P$ in $\Coh^{\bar{\beta}}(X)$. Since $P$ is $\sigma^0_{\bar{\alpha},\bar{\beta}}$-semistable, we see  $\cH^0(P)$ is either $0$ or supported on points. Moreover, $\cH^{-1}(P)$ is $\sigma_{\bar{\alpha},\bar{\beta}}$-semistable. Denote $\ch(\cH^{-1}(P))=(-n,nH,-\frac{nH^2}{2},m)$. By \cite[Conjecture 4.1]{BMS}\cite{Li}, $m\leqslant n\frac{H^3}{6}$. Note $\Delta(\cH^{-1}(P))=0$. By \cite[Corollary 3.10]{BMS}, either $\cH^{-1}(P)$ is $\sigma_{\bar{\alpha},\bar{\beta}}$-stable or it is strictly $\sigma_{\bar{\alpha},\bar{\beta}}$-semistable and its Jordan-H\"older factors have $\ch_{\leqslant 2}$ proportional to $\ch_{\leqslant 2}(\cH^{-1}(P))$. Let $R_i$ ($1\leqslant i\leqslant N$) be the factors in a Jordan-H\"older filtration of $\cH^{-1}(P)$, then $\ch(R_i)=(-k_i,k_iH,-\frac{k_iH^2}{2},r_i)$ for positive integers $k_i$ such that $k_1+\cdots+k_N=n$. By \cite[Corollary 3.11(c)]{BMS} and using $\H^j(R_i)$ to denote the $j$-th cohomology of $R_i$ in $\Coh(X)$, $\H^0(R_i)$ has zero dimensional support (say of length $l_i$) and $\H^{-1}(R_i)$ is a slope-semistable sheaf with Chern character $(k_i,-k_iH,\frac{k_iH^2}{2},-r_i+l_i)$. Then $\H^{-1}(R_i)\otimes \cO_X(1)$ is a slope-semistable sheaf with Chern character $(k_i,0,0,k_i\frac{H^3}{6}-r_i+l_i)$. By \cite[Proposition 4.18(i)]{BBF et},
\begin{align}\label{inequi}
    k_i\frac{H^3}{6}-r_i+l_i\leqslant 0
\end{align}
 for each $1\leqslant i\leqslant N$. Note 
\begin{align*}
    \sum_{i=1}^N r_i=m\leqslant n\frac{H^3}{6}
\end{align*}
Combined with (\ref{inequi}), we have $\sum_{i=0}^Nl_i=0$. Since all $l_i$ are nonnegative integers, we have $l_i=0$ and $r_i=k_i\frac{H^3}{6}$ for all $i$. By \cite[Proposition 4.18(i)]{BBF et}, we see $R_i\cong \cO_X(-1)^{\oplus k_i}[1]$ for all $1\leqslant i\leqslant N$. Hence $\cH^{-1}(P)=\cO_X(-1)^{\oplus n}[1]$. Next we claim $\cH^0(P)=0$. Suppose otherwise, then we have a sequence 
\begin{align*}
    0\to \cO_X(-1)^{\oplus n}[2]\to G\to Q'\to 0
\end{align*}
in $\Coh^0_{\bar{\alpha},\bar{\beta}}(X)$, where $Q'$ is defined by $0\to \cH^0(P)\to Q'\to Q\to 0$. Since $G\in \mathsf{Ku}(X)$ and $H^i(\cO_X(-1))=0$ for all $i$, we have $\Hom(\cO_X,Q')=0$. We have the long exact sequence
\begin{align*}
    \to \Hom(\cO_X,Q[-1])\to \Hom(\cO_X,\cH^0(P))\to \Hom(\cO_X,Q')\to
\end{align*}
Note the first term is $0$ since $\cO_X$ and $Q$ are in the same heart, thus $\Hom(\cO_X,\cH^0(P))=0$ and in turn $\cH^0(P)=0$, i.e. $P=\cO_X(-1)^{\oplus n}[2]$.

We can assume that $n$ is maximal, that is, $\cO_X(-1)[2]$ is not a subobject of $Q$ in $\Coh^0_{\bar{\alpha},\bar{\beta}}(X)$.
It is easy to compute that $\chi(\cO_X(-1)[2],Q)=-2d-n$ and $\chi(P,Q)=-2nd-n^2<0$. Note $\Hom(P,Q[i])$ is $0$ if $i<0$ or $i>3$. We have $$\Hom(P,Q[3])=\Hom(Q,\cO_X(-3)^{\oplus n}[2])=0.$$ Then $\Hom(P,Q[1])>0$. Moreover, since $\mathrm{dim}( \Hom(\cO_X(-1)[2],Q[1]))\geqslant2d+n>n$, we can define $G'$ as a extension
\begin{align*}
    0\to Q\to G'\to P\to 0
\end{align*}
in $\Coh^0_{\bar{\alpha},\bar{\beta}}(X)$ which is determined by $n$ linearly independent vectors in $\Hom(\cO_X(-1)[2],Q[1])$.\ 

We claim $G'$ is $\sigma^0_{\alpha,\beta}$-semistable above the wall $\mathcal{C}$ in a neighbourhood of $(\bar{\alpha},\bar{\beta})$. By our previous three lemmas and the first paragraph of this proof, it suffices to show $\cO_X(-1)[2]$ is not a subobject of $G'$ in $\Coh^0_{\bar{\alpha},\bar{\beta}}(X)$. Suppose otherwise. Since $n$ is maximal, the induced map from the subobject $\cO_X(-1)[2]$ to $P=\cO_X(-1)^{\oplus n}[2]$ is nontrivial. 
 However, the composition of this  induced map $O_X(-1)[2]\to P$ with $P\to Q[1]$(defined by $G'$) is trivial, which contradicts our construction of $G'$.

Since $\Hom(\cO_X,G'[i])=0$ for all $i$, by Lemma \ref{nowall-1}, $G'$ is $\sigma^0_{\alpha,\beta}$-semistable for $\beta\to-\half$. We can argue as in the previous case and conclude that $G'=F'[1]$ where $F'$ is a Gieseker-semistable sheaf with Chern character $(2,0,-2,0)$. Then $F'\in \mathsf{Ku}(X)$, so $\Hom(G',P)=0$ gives a contradiction.
\end{proof}
\begin{rem}
    We note the same arguments work for case (1) until the last line when we conclude $F'\in \mathsf{Ku}(X)$. See Remark \ref{deg5vanishing}.
\end{rem}

\begin{thm}\label{MOD}
    Let $X$ be a Fano threefold of Picard rank one, index two and degree $d$. If $d\geqslant3$, then for any $\sigma\in\mathcal{K}$, the moduli space of Gieseker-semistable sheaves on $X$ with Chern character $(2,0,-2,0)$ and satisfying $H^1(E(-1))=0$ is isomorphic to a moduli space $M_\sigma(\mathsf{Ku}(X),2[\cI_l])$ of $\sigma$-semistable objects in $\mathsf{Ku}(X)$ with numerical class twice of that of an ideal sheaf of a line in $X$.
\end{thm}
\begin{proof}
    By Propositions \ref{stab}, \ref{conicstab}, \ref{stabrev}, Corollary \ref{stabcor} and Remark \ref{semistab}, we see the notion of Gieseker-semistable sheaves on $X$ with Chern character $(2,0,-2,0)$ and satisfying $H^1(E(-1))=0$ is the same as $\sigma$-semistable objects in $\mathsf{Ku}(X)$ with numerical class $2[\cI_l]$ up to shifts. Moreover, the $S$-equivalences are compatible. As a result, having a family of Gieseker-semistable sheaves on $X$ with Chern character $(2,0,-2,0)$ and satisfying $H^1(E(-1))=0$ is equivalent to having a family of $\sigma$-semistable objects in $\mathsf{Ku}(X)$ with numerical class $2[\cI_l]$ up to shifts. We can then identify their respective moduli functors. Since the first functor is co-represented by the moduli spaces in Theorem \ref{moduli}, so is the second.
\end{proof}
\begin{rem}
    For $d=4$, there is another way to understand the previous theorem. Recall there is an equivalence $\Ku(X)\cong\mathrm{D}^b(C)$ for a smooth projective curve $C$ of genus $2$ (see \cite[Section 5]{Ku2} for the precise definition of the equivalence). By \cite[Theorem 2.7]{Ma}, we have $\mathcal{K}\cong\mathrm{Stab}(\Ku(X))\cong \mathrm{Stab}(C)=\sigma_0\cdot\tilde{\mathrm{GL}}_2^+(\R)$, where $\sigma_0$ is the slope stability on the curve $C$. Then for any $\sigma\in\mathcal{K}$, $\sigma$-semistable objects in $\mathsf{Ku}(X)$ with numerical class $2[\cI_l]$ corresponds to semistable vector bundles of rank $2$ and degree $0$ on $C$. The previous theorem for $d=4$ will then follow from \cite[Theorem 1.5]{Q2}.
\end{rem}

\section{Non-minimal instanton bundles}
In this section, we explore some applications of our methods from the previous sections to non-minimal instanton bundles. Let $\cE$ be an instanton bundle and let $n$ be its charge. Then $\cE$ is an object in $\Ku(X)$ if and only if $n=2$. On the other hand, one can associate to $\cE$ a unique vector bundle of rank $n$ which is an object in $\Ku(X)$:
\begin{lem}\cite[Lemma 3.5 and 3.6]{Ku2}
For each instanton bundle $\cE$ there exists a unique short exact sequence
\begin{align}\label{acyclic}
    0\to \cE\to\tilde{\cE}\to \cO_X^{n-2}\to 0,
\end{align}
where $\tilde{\cE}\in\Ku(X)$ is a simple slope-semistable vector bundle with Chern character $\ch(\tilde{\cE})=(n,0,-n,0)$. $\tcE$ is called the acyclic extension of $\cE$.
\end{lem}
\begin{rem}
    As remarked in \cite{Ku2}, $\tcE$ is nothing but the universal extension of $H^1(E)\otimes\cO_X$ by $\cE$. It can be also viewed as the left mutation $\mathbb{L}_{\cO_X}\cE$ of $\cE$ through $\cO_X$. It is clear that if $n=2$, $\cE=\tcE$.
\end{rem}


The next lemma and its proof are suggested by an anonymous referee.
\begin{lem}\label{acyc2}
    Let $E\ncong \cO_X$ be a slope-stable vector bundle with $\mu(E)=0$ on $X$. Let $\tE$ be the vector bundle defined by the universal extension
    \begin{align*}
        0\to E\to \tE\to H^1(E)\otimes\cO_X\to 0.
    \end{align*}
    Then $\tE$ is Gieseker-stable.
\end{lem}
\begin{proof}
    Since $E$ is slope-stable,  $\ch_2(E)\leq 0$ by Bogomolov inequality. If $\ch_2(E)=0$, by \cite[Proposition 4.18(i)]{BBF et} and our assumption that $E\ncong\cO_X$, $\ch_3(E)<0$. In particular $P(E)\prec P(\cO_X)$ and $h^0(E)=0$.
    We note that $\tE$ is a slope-semistable vector bundle with $\mu(\tE)=0$. Moreover, $h^0(\tE)=h^1(\tE)=0$.\ 
    
    Suppose $\tE$ is not Gieseker-stable. Then $H^1(E)\neq0$ and $P(E)\prec P(\tilde{E})$. Let $0\neq G\hookrightarrow\tE$ be a Gieseker-stable subsheaf so that $P(G)\succeq P(\tE)$. Then $\mu(G)=0$. We can further assume that $\tE/G$ is torsion free. Since $\tE$ is locally free, $G$ is reflexive. Let $K$ and $I$ be the kernel and image of the composite morphism:
    \begin{align*}
        f: G\hookrightarrow\tE\to H^1(E)\otimes\cO_X
    \end{align*}
    respectively. If $I=0$, then $G$ is a nontrivial proper subsheaf of $E$ with $\mu(G)=0$, this contradicts the assumption that $E$ is slope-stable, thus $I\neq 0$. As $I$ is a subsheaf of $H^1(E)\otimes\cO_X$ as well as a quotient sheaf of $G$, $\mu(I)=0$, thus either $\mu(K)=0$ or $K=0$. Now $K$ is a subsheaf of $E$, if $K=E$, then one easily checks that $P(G)\prec P(\tilde{E})$, contradicting our choice of $G$. So we must have $K=0$, and $f$ is injective.\ 
    
    If $G\hookrightarrow H^1(E)\otimes \cO_X$ is saturated, then \cite[Corollary 1.6.11]{HL} implies that $G\cong\cO_X$, which contradicts the fact that $G$ is a subsheaf of $\tE$ and $h^0(\tE)=0$.\ 
    
    If $G\hookrightarrow H^1(E)\otimes \cO_X$ is not saturated, let $G'$ be its saturation. We have the following commutative diagram
            \[\begin{tikzcd}
         &&&0\ar{d}&\\
    &&&T\ar{d}&\\
    0\arrow{r} &G\arrow{r}\arrow[hook]{d}&H^1(E)\otimes \cO_X\arrow{r}\arrow[equal]{d}&Q\arrow{r}\ar{d}&0\\
    0\ar{r}&G'\arrow{r}&H^1(E)\otimes \cO_X\ar{r}&Q'\ar{r}\ar{d}&0\\
    &&&0
    \end{tikzcd}
    \]
    where $T$ is the torsion part of $Q$, and the short exact sequence
    \begin{align*}
        0\to G\to G'\to T\to 0.
    \end{align*}
    Since $H^1(E)\otimes \cO_X$ is locally free and $Q'$ is torsion free, $G'$ is reflexive. If $\dim(T)=2$, then $c_1(T)>0$. Since $Q'$ is a quotient of $H^1(E)\otimes \cO_X$, $c_1(Q')\geq 0$. Now $0=c_1(Q)=c_1(Q')+c_1(T)>0$ leads to a contradiction. It follows that $\dim(T)\leqslant1 $ and $\mathcal{E}xt^1(T,\cO_X)=0$. Dualizing the short exact sequence above we obtain:
    \begin{align*}
        0\to \mathcal{E}xt^1(G',\cO_X)\to \mathcal{E}xt^1(G,\cO_X)\to \mathcal{E}xt^2(T,\cO_X)\to 0
    \end{align*}
    and $\mathcal{E}xt^3(T,\cO_X)=0$, i.e $T$ must have pure dimension $1$. Since $G$ and $G'$ are reflexive, $\mathcal{E}xt^1(G',\cO_X)$ and $ \mathcal{E}xt^1(G,\cO_X)$ are both $0$-dimensional (\cite[Proposition 1.1.10]{HL}). But $\mathcal{E}xt^2(T,\cO_X)$ is $1$-dimensional, which leads to a contradiction.
\end{proof}

\begin{cor}\label{acyc3}
    Let $\cE$ be an instanton bundle on $X$. Then its acyclic extension $\tcE$ is Gieseker-stable.
\end{cor}
\begin{proof}
    By \cite[Lemma 1.23]{Sa}, $\cE$ is slope-stable. Since $\tcE$ is precisely the universal extension of $H^1(\cE)\otimes \cO_X$ by $\cE$, we conclude by Lemma \ref{acyc2}.
\end{proof}

One can now ask for $n\geqslant3$, whether $\tcE\in\Ku(X)$ is $\sigma$-stable with respect to $\sigma\in\mathcal{K}$. We show that the answer to this question is positive for $n=3$.

\begin{prop}\label{extend}
    Assume $d\neq 1$. Let $\tcE$ be the acyclic extension of an instanton bundle $\cE$ of charge $3$. Then $\tcE$ is $\sigma$-semistable for any stability condition $\sigma\in\mathcal{K}$.
\end{prop}
\begin{proof}
    By Corollary \ref{acyc3} and \cite[Proposition 4.8]{BBF et}, $\tcE$ is $\sigma_{\alpha,-\half}$ semistable for $\alpha\gg0$. We proceed as the proof of Proposition \ref{stab} to show there is no wall that will make $\tcE$ tilt-unstable along $\beta=-\half$. A wall would be given by a sequence in $\Coh^{-\half}(X)$:
    \begin{align*}
        0\to\tcE'\to\tcE\to\tcE''\to0
    \end{align*}
    in which the truncated twisted Chern characters satisfy
    \begin{align*}
        (3,\frac{3}{2}H,\frac{3}{8}(d-8))=(a,\frac{b}{2}H,\frac{c}{8})+(3-a,\frac{3-b}{2}H,\frac{3d-24-c}{8})
    \end{align*}
    for some $a,b,c\in\Z$. As in Proposition \ref{stab}, the wall condition and Bogomolov inequality imply
    \begin{enumerate}
    \item $\frac{1}{b}(\frac{c}{4d}-\alpha^2a)=\frac{d-8-4d\alpha^2}{4d}=\frac{1}{3-b}(\frac{3d-24-c}{4d}-\alpha^2(3-a))$;
    \item $(\frac{b}{2})^2-\frac{ac}{4d}\geqslant 0$, $(\frac{3-b}{2})^2-\frac{(3-a)(3d-24-c)}{4d}\geqslant 0$
    \item $(\frac{b}{2})^2-\frac{ac}{4d}\leqslant \frac{18}{d}$, $(\frac{3-b}{2})^2-\frac{(3-a)(3d-24-c)}{4d}\leqslant \frac{18}{d}$. 
\end{enumerate}
     Since $\tcE',\tcE''$ are in $\Coh^{-\half}(X)$, we have $b=0,1,2$ or $3$. We can easily eliminate the cases of $b=0$ and $b=3$ as in the proof of Proposition \ref{stab}. Note $\frac{b-a}{2}H$ is the first Chern character of either $\tcE'$ or $\tcE''$. Hence $a,b$ have the same parity. One of $a$ and $3-a$ will be at least two. Without loss of generality, we assume $a\geqslant2$.\ 
    
    Suppose $b=2$. Condition (1) implies
    $c=2d-16+4d\alpha^2(a-2)$.
    We observe that if $a=2$, then $(a,\frac{b}{2}H,\frac{c}{8})=(2,H,\frac{2d-16}{8})$ is proportional to the truncated twisted character of $\tcE$. This case will not affect tilt-semistability and can be ignored. If $a\geqslant 4$, the second inequality of (2) simplifies to
    \begin{align*}
        4d\alpha^2(a-2)\leqslant \frac{d}{a-3}+d-8\leqslant 2d-8.
    \end{align*}
    This immediately leads to a contradiction for $d\leqslant4$ since the left hand side is positive while the right hand side is non-positive. For $d=5$, the above equation becomes
    \begin{align*}
        \alpha^2\leqslant\frac{1}{10(a-2)}\leqslant\frac{1}{20}.
    \end{align*}
    In this case there are no walls for $\alpha>\frac{1}{\sqrt{20}}$.\ 
    
    Suppose $b=1$. Condition (1) implies $c=d-8+4d\alpha^2(a-1)$. If $a\geqslant5$, the second inequality of (2) simplifies to 
    \begin{align*}
        4d\alpha^2(a-1)\leqslant\frac{4d}{a-3}+2d-16\leqslant4d-16.
    \end{align*}
    Arguing as above we see there are no walls when $d\leqslant4$ and no wall with $\alpha>\frac{1}{\sqrt{20}}$ when $d=5$. If $a=3$, the formula for $c$ along with the first inequality of (2) implies $d-8<c\leqslant\frac{d}{3}$. Moreover, using the formula for twisted Chern character, $c=3d-8c_2$ where $c_2\in\Z$ is the second Chern class of either $\cE'$ or $\cE''$. Hence $c\equiv 3d \pmod{8}$. As a result, we have:
    \begin{itemize}
        \item for $d=5$, $c=-1$ and $\alpha=\frac{1}{\sqrt{20}}$;
        \item for $d=4$, no solution;
        \item for $d=3$, $c=1$ and $\alpha=\half$;
        \item for $d=2$, $c=-2$ and $\alpha=\half$.
    \end{itemize}
    Next we eliminate the cases when $d=3$ and $2$. For $d=3$, either the destabilizing subobject or quotient, which we denote by $\mathcal{G}$, has truncated twisted character $(3,\half H,\frac{1}{8})$. Then $\Delta(\mathcal{G})=0$. By \cite[Corollary 3.10]{BMS}, $\mathcal{G}$ is $\sigma_{\half,-\half}$-stable. Now we can exclude this case by Proposition \ref{Libound}(3). For $d=2$, either the destabilizing subobject or quotient, which we denote by $\mathcal{G}'$, has truncated twisted character $(3,\half H,-\frac{1}{4})$. Then it is easy to see $\mathcal{G}'$ is $\sigma_{\half,-\half}$-stable by checking the twisted first Chern class of its subobjects and quotients. Now we can exclude this case by Proposition \ref{Libound}(5) and the fact that the rank of $\mathcal{G}'$ is $3$.\ 
    
    To summarize, for $2\leqslant d\leqslant4$, there are no walls on $\beta=-\half$; for $d=5$, there are no walls with $\alpha>\frac{1}{\sqrt{20}}$ on $\beta=-\half$. We conclude using the argument at the end of the proof for Proposition \ref{stab}.
\end{proof} 


\begin{cor}\label{ext2}
Assume $d\geqslant3$. The acyclic extension $\tcE$ of an instanton bundle $\cE$ of charge $3$ is $\sigma$-stable for any stability condition $\sigma\in \mathcal{K}$.
\end{cor}
\begin{proof}
    It remains to show $\tcE$ is not strictly $\sigma$-semistable. Suppose otherwise, by \cite[Theorem 1.1]{PY} and Theorem \ref{MOD}, the Jordan-H\"older factors (with respect to $\sigma$) of $\tcE$ are either ideal sheaves of lines on $X$ or rank $2$ semistable sheaves with Chern class $c_1=0$, $c_2=2$ and $c_3=0$. It suffices to show none of these sheaves can have nonzero morphisms to $\tcE$. Let $l\subset X$ be a line. We have long exact sequence
    \begin{align*}
        \to \Hom(\cO_X,\tcE)\to \Hom(\cI_l,\tcE)\to \Ext^1(\cO_l,\tcE)\to
    \end{align*}
    It is straightforward to check the left and right term both vanish, thus $\Hom(\cI_l,\tcE)=0$.\ 
    
    Let $E$ be a minimal instanton bundle. We have long exact sequence
    \begin{align*}
        \to \Hom(E,\cE)\to \Hom(E,\tcE)\to \Hom(E,\cO_X)\to 
    \end{align*}
    The left term is $0$ by Gieseker-stability, while the right term is $H^0(X,E)=0$ since $E$ is self-dual. Thus $\Hom(E,\tcE)=0$. \ 
    
    Let $E'$ be a stable but non-locally free sheaf with Chern character $(2,0,-2,0)$. By Proposition \ref{classfication}, we have long exact sequence
    \begin{align*}
        \to\Hom(\cO_X^{\oplus2},\tcE)\to\Hom(E',\tcE)\to \Ext^1(\theta(1),\tcE)\to
    \end{align*}
    where $\theta$ is the theta character of a smooth conic in $X$.    It is straightforward to check the left and right term both vanish, thus $\Hom(E',\tcE)=0$.\

\end{proof}

\begin{rem}
    Unfortunately the author was unable to extend Proposition \ref{extend} and Corollary \ref{ext2} to the cases when either $d=1$ or charge $n\geqslant4$. In each of these cases, there will be potential walls which we fail to eliminate.
\end{rem}

\end{document}